\documentclass[11pt, psamsfonts, reqno]{amsart}
\usepackage{amssymb,amsmath,amsthm,mathrsfs}
\usepackage{amscd}
\usepackage{amsfonts}
\usepackage{amsbsy}
\usepackage{epsfig}
\usepackage{graphicx}
\usepackage{psfrag}
\usepackage{a4wide}
\usepackage{bbm}

\newtheorem{theorem}{Theorem}
\newtheorem{remark}{Remark}
\newtheorem{proposition}{Proposition}

\newtheorem*{condition}{Condition}

\newenvironment{definition}{{\bf Definition.}}{}

\usepackage[T2A,T1]{fontenc}
\DeclareSymbolFont{cyrillic}{T2A}{cmr}{m}{n}
\DeclareMathSymbol{\D}{\mathalpha}{cyrillic}{196}

\def\S{{\mathcal S}}
\def\ie{{\em i.e.,} }

\newfont\bbf{msbm10 at 12pt}
\def\eps{\varepsilon}
\def\phi{\varphi}
\def\R{{\mathbb R}}

\def\N{{\mathbb N}}

\def\B{{\mathcal B}}
\def\M{\ensuremath{\mathcal M}}
\def\I{{\mathbbm{1}}}
\def\P{{\mathcal P}}

\def\p{\ensuremath{\mathbb P}}
\def\A{\ensuremath{A^{(q)}}}

\def\X{{\mathcal X}}

\def\sm{\setminus}

\def\nn{\ensuremath{\mathscr N}}
\def\RR{\ensuremath{\mathcal R}}
\def\cyl{Z}

\def\le{\leqslant}
\def\ge{\geqslant}

\newcommand{\z}{\zeta}
  \def\a{\alpha} 
\def\e{\ensuremath{\text{e}}}
\def\bH{\bar H}

\begin{document}

\title{Rare Events for the Manneville-Pomeau map}

\author[A. C. M. Freitas]{Ana Cristina Moreira Freitas}
\address{Ana Cristina Moreira Freitas\\ Centro de Matem\'{a}tica \&
Faculdade de Economia da Universidade do Porto\\ Rua Dr. Roberto Frias \\
4200-464 Porto\\ Portugal} \email{amoreira@fep.up.pt}
\urladdr{http://www.fep.up.pt/docentes/amoreira/}

\author[J. M. Freitas]{Jorge Milhazes Freitas}
\address{Jorge Milhazes Freitas\\ Centro de Matem\'{a}tica da Universidade do Porto\\ Rua do
Campo Alegre 687\\ 4169-007 Porto\\ Portugal}
\email{jmfreita@fc.up.pt}
\urladdr{http://www.fc.up.pt/pessoas/jmfreita}

\author[M. Todd]{Mike Todd}
\address{Mike Todd\\ Mathematical Institute\\
University of St Andrews\\
North Haugh\\
St Andrews\\
KY16 9SS\\
Scotland \\} \email{m.todd@st-andrews.ac.uk }
\urladdr{http://www.mcs.st-and.ac.uk/~miket/}

\author{ Sandro Vaienti}
\address{Sandro Vaienti\\
Aix Marseille Universit\'e, CNRS, CPT, UMR 7332, 13288 Marseille, France and Universit\'e de Toulon, CNRS, CPT, UMR
7332, 83957 La Garde, France} 
\email {vaienti@cpt.univ-mrs.fr}
\urladdr{http://www.cpt.univ-mrs.fr/~vaienti/}

\thanks{ACMF was partially supported by FCT grant SFRH/BPD/66174/2009. JMF was partially supported by FCT grant SFRH/BPD/66040/2009. Both these grants are financially supported by the program POPH/FSE. ACMF, JMF are partially  supported by FCT (Portugal) project FAPESP/19805/2014. All authors are partially  supported by FCT (Portugal) projects PTDC/MAT/120346/2010 and PTDC/ MAT-CAL/3884/2014. All these projects are funded by national and European structural funds through the programs  FEDER and COMPETE. ACMF, JMF and MT are also supported by CMUP (UID/MAT/00144/2013), which is funded by FCT (Portugal) with national (MEC) and European structural funds through the programs FEDER, under the partnership agreement PT2020. SV has been supported by the ANR grant PERTURBATIONS and by the project MOD TER COM of the french Region PACA}



\begin{abstract}
We prove a dichotomy for Manneville-Pomeau maps $f:[0,1]\to [0, 1]$:  given any point $\z\in [0,1]$, either the Rare Events Point Processes (REPP), counting the number of exceedances, which correspond to entrances in balls around $\z$, converge in distribution to a Poisson process; or the point $\z$ is periodic and the REPP converge in distribution to a compound Poisson process.  Our method is to use inducing techniques for all points except 0 and its preimages, extending a recent result \cite{HayWinZwe14}, and then to deal with the remaining points separately. The preimages of 0 are dealt with applying recent results in \cite{AFV15}. The point $\zeta=0$ is studied separately because the tangency with the identity map at this point creates too much dependence, which causes severe clustering of exceedances. The Extremal Index, which measures the intensity of clustering, is equal to 0 at $\zeta=0$, which ultimately leads to a degenerate limit distribution for the partial maxima of stochastic processes arising from the dynamics and for the usual normalising sequences. We prove that using adapted normalising sequences we can still obtain non-degenerate limit distributions at $\zeta=0$.  
\end{abstract}

\date{\today}

\maketitle

\section{Introduction}

One of the standard ways to investigate the statistical properties of a dynamical system $f:\X\to \X$ with respect to some measure $\p$ is to look at its recurrence to certain points $\z$ in the system.  This can be connected to Extreme Value theory: supposing that $\phi:\X\to \R$ is an observable taking its unique maximum $u_\phi$ at $\z$, one can look at the behaviour of the iterates $x, f(x), f^2(x), \ldots$ via the observations
\begin{equation}
\label{eq:def-SP}
X_i=X_i(x)=\phi\circ f^n(x).
\end{equation}
If $\p$ is an $f$-invariant probability measure then $X_0, X_1, \ldots$ is a stationary stochastic process.  We furthermore assume that $\p$ is ergodic in order to isolate specific statistical behaviour.  So Birkhoff's Ergodic Theorem implies that these random variables will satisfy the law of large numbers.
We can now consider the random variable given by the maximum of this process:
$$M_n=M_n(x):=\max\{X_0(x), \ldots, X_{n-1}(x)\}.$$
Again by the ergodic theorem, if $\z$ is in the support of $\p$ then for any small ball around $\zeta$, $f^n(x)$ must hit the ball for some $n$ and typical $x$.  Hence if $\phi$ is sufficiently regular, then we expect $M_n\to u_\phi$.  Therefore, to obtain a non-trivial limit law, we need to rescale $\{M_n\}_n$.
Indeed, we say that we have an \emph{Extreme Value Law} (EVL) for $M_n$ if there is a non-degenerate distribution function $H:\R\to [0,1]$ with $H(0)=0$ and a sequence of levels $u_n=u_n(\tau)$ such that
\begin{equation}
\label{eq:un}n\p(X_0>u_n)\to \tau \text{ as } n\to \infty,\end{equation}
and for which the following holds:
$$\p(M_n\le u_n) \to \bH(\tau) = 1-H(\tau) \text{ as } n\to \infty,$$
where the convergence is meant at the continuity points of $H(\tau)$.  

In recent years, there has been a great deal of work on EVLs in the context of dynamical systems (see for example
\cite{C01, FF08, VHF09, FreFreTod10, GHN11, HNT12, LFW12, K12, FHN14, AFV15}),
the standard form of the observable $\phi$ being a function of the distance to $\z$, for example $\phi(x)=-\log d(x, \z)$ for $d$ a metric on $\X$.  Note that instead of the $\log$, different functions can be composed with the distance  (see \cite[page 679]{FreFreTod10}); moreover, $\phi$ need not depend on the distance to $\z$, see \cite[Section~5]{FFT11}. However, for the purposes of this paper, as in \cite{FreFreTod10}, we assume that $\phi$ is a function of the distance to $\zeta$ and is invertible in a vicinity of $\zeta$, so that $\varphi^{-1}(u_n)$ is well defined for $u_n$ sufficiently close to $\varphi(\zeta)$, which can possibly be infinite (see Section~\ref{ssec:shapes} for further comments on this).

In many cases it has been shown that for $\p$-a.e. $\z\in \X$, this setup gives an EVL with $\bH=e^{-\tau}$.  More recently it has been shown that if $\z$ is a periodic point of period $p$ then $\bH=e^{-\theta\tau}$ where $\theta\in (0,1)$ depends on the Jacobean of the measure for $f^p$, and is referred to as the \emph{Extremal Index} (EI). The EI is known to measure the intensity of clustering of exceedances of the levels $u_n$. In fact, in many cases, the EI is equal to the inverse of the average cluster size, so that the EI is equal to 1 when there is no clustering. In the case of a class of uniformly hyperbolic dynamical systems, a stronger property, a dichotomy, has been shown: either $\z$ is periodic and we have an EVL with some extremal index $\theta\in (0,1)$, or there is an EVL $\bH=e^{-\tau}$.  This was shown for $f$ some uniformly expanding interval map with a finite number of branches in \cite{FerPol12} (see also \cite[Section 6]{FreFreTod12}) and with a countable number of branches in \cite{AFV15}; here, depending on the precise form of the map, the measure can be absolutely continuous with respect to Lebesgue (acip), or an equilibrium state for some H\"older potential.  Inducing methods have been used to extend some of these results to non-uniformly hyperbolic dynamical systems (see \cite{FreFreTod13} which built on \cite{BruSauTroVai03}), but the results have not thus far extended to such a complete dichotomy.

We note that from a more probabilistic direction, such as processes under certain mixing conditions, which can be related to some symbolic dynamical systems, there is related work on Hitting Time Statistics, which as in \cite{FreFreTod10} can be seen as Extreme Value Laws.  For a review on early work in this direction see \cite{AbaGal01}.  For recent work, including some analysis of \emph{all} points, including periodic ones, see \cite{AbaVer09, AbaSau11}.

We can further enrich our process by considering the point process formed by entries into the regions $\{X>u_n\}$, which in good cases gives rise to a Poisson process.  An analogous dichotomy can often be shown there also: in the case of a periodic point $\z$, we obtain a compound Poisson process.  We leave the details of this construction to later.

In this note, we extend the dichotomy to a simple non-uniformly hyperbolic dynamical system, the \emph{Manneville-Pomeau} (MP) map equipped with an absolutely continuous invariant probability measure.  The particular form of these maps we will use is, for $\alpha>0$,
\begin{equation*}
f=f_\alpha(x)=\begin{cases} x(1+2^\alpha x^\alpha) & \text{ for } x\in [0, 1/2)\\
2x-1 & \text{ for } x\in [1/2, 1]\end{cases}
\end{equation*}
Members of this family of maps are often referred to as Liverani-Saussol-Vaienti maps since they were defined in \cite{LSV99}.  If $\alpha\in (0,1)$ then there is an acip $\mu_\alpha$: we will restrict our attention to this case.
As can be seen for example in \cite{LSV99,Y99,H04}, for each $\alpha\in (0,1)$, the system $([0,1], f_\alpha, \mu_\alpha)$ has polynomial decay of correlations.  That is, letting $\mathcal H_\beta$ denote the space of H\"older continuous functions $\phi$ with exponent $\beta$ equipped with the norm $\|\phi\|_{\mathcal H_\beta}=\|\phi\|_\infty+|\phi|_{\mathcal H_\beta}$, where $$|\phi|_{\mathcal H_\beta}=\sup_{x\neq y}\frac{|\phi(x)-\phi(y)|}{|x-y|^\beta},$$ 
 there exists $C>0$ such that for each $\phi\in \mathcal H_\beta$, $\psi\in L^\infty$ and all $t\in \N$,
\begin{equation}
\label{eq:Holder-DC}
\left| \int\phi\cdot(\psi\circ f^t)d\mu_\alpha-\int\phi d\mu_\alpha\int\psi
d\mu_\alpha\right|\leq C\|\phi\|_{\mathcal H_\beta}\|\psi\|_\infty \frac{1}{t^{\frac1\alpha-1}}.
\end{equation}
Let $h_\alpha=\frac{d\mu_\alpha}{dx}$. From \cite{H04} we know that $h_\alpha\in L^{1+\epsilon}$, with $\epsilon<1/\alpha-1$, and moreover $\lim_{x\to0}\frac{h(x)}{x^{-\alpha}}=C_0>0$. Hence, for small $s>0$ we have that 
\begin{equation}
\label{eq:estimate-measure}
\mu_\alpha([0,s))\sim_c s^{1-\alpha},
\end{equation}
where the notation $\sim_{c}$ is used in the sense that there is $c>0$ such that $\lim_{s\to\infty}\frac{\mu_\alpha([0,s))}{s^{1-\alpha}}=c$.

In this case there are canonical induced maps which capture all but a countable number of points in the phase space, so with some extra consideration for those points not captured, we can prove the full dichotomy, where for $\z$ a periodic point of period $p$, the extremal index is $\theta=1-1/|Df^p(\z)|$.  

For the special case in which $\z$ is the indifferent fixed point, we prove that there exists an EI equal to zero, which corresponds to a degenerate limit law, when the usual normalising sequences are used.  Moreover, for a particular range of $\alpha$ we show that by changing the definition of $(u_n)_n$ given by \eqref{eq:un} in a suitable way, we recover a non-degenerate EVL.  This latter result relies on information on the transfer operator in \cite{HirSauVai99} as well as a refinement of the techniques for proving EVLs at periodic points developed in \cite{FFT15}.

The results that we present here for the Liverani-Saussol-Vaienti maps should be extendable to more general models of intermittent maps with neutral fixed (or periodic) points as long as they admit a first return time induced map which falls in the category of the uniformly expanding maps studied by Rychlik in \cite{R83}. We chose not to treat more general models because we believe that in this way the ideas are presented in a much easier way, without an unnecessary overload of notation and length.

\subsection{Point process of hitting times}   We will use our observations on our dynamical system to generate point processes. Here we adopt the approach and notation of \cite{Zwe07a}.
Let $\M_p([0,\infty))$ be the space of counting measures on $([0,\infty), \B_{[0,\infty)})$.  We equip this space with the vague topology, i.e., $\nu_n\to \nu$ in $\M_p([0,\infty))$ whenever $\nu_n(\psi)\to \nu(\psi)$ for any continuous function $\psi:[0,\infty)\to \R$ with compact support.  A \emph{point process} $N$ on $[0,\infty)$ is a random element of $\M_p([0,\infty))$.  We will be interested in point processes $N_n:X\to \M_p([0,\infty))$.  If we have a fixed measure $\mu$ on $X$, we say that $(N_n)_n$ \emph{converges in distribution} to $N$ if $\mu\circ N_n^{-1}$ converges weakly to $\mu\circ N^{-1}$.  We write $N_n \stackrel{\mu}{\Longrightarrow} N $.

So given $X_0, X_1, X_2, \ldots$ and some $u\in \R$, we begin the construction of our point process $\R\to \M_p([0,\infty))$ as follows.  Given   $A\subset\R$  we define 
\[
\nn_u(A):=\sum_{i\in A\cap\N_0}\I_{X_i>u}.
\]
So $\nn_{u}[0,n)$ counts the number of exceedances of the parameter $u$ among the first $n$ observations of the process $X_0,X_1,\ldots,X_{n-1}$ or, in other words, the number of entrances in $U(u):=\{X_0>u\}$ up to time $n$.

We next re-scale time using the factor $v:=1/\p(X>u)$ suggested by Kac's Theorem. However, before we give the definition, we need some formalism. Let $\S$ denote the semi-ring of subsets of  $\R_0^+$ whose elements
are intervals of the type $[a,b)$, for $a,b\in\R_0^+$. Let $\RR$
denote the ring generated by $\S$. Recall that for every $J\in\RR$
there are $k\in\N$ and $k$ intervals $I_1,\ldots,I_k\in\S$ such that
$J=\cup_{i=1}^k I_j$. In order to fix notation, let
$a_j,b_j\in\R_0^+$ be such that $I_j=[a_j,b_j)\in\S$. For
$I=[a,b)\in\S$ and $\alpha\in \R$, we denote $\alpha I:=[\alpha
a,\alpha b)$ and $I+\alpha:=[a+\alpha,b+\alpha)$. Similarly, for
$J\in\RR$ define $\alpha J:=\alpha I_1\cup\cdots\cup \alpha I_k$ and
$J+\alpha:=(I_1+\alpha)\cup\cdots\cup (I_k+\alpha)$.

We suppose that $\tau>0$ and $(u_n)_{n}$ is defined so that \eqref{eq:un} holds.  We let $U(u_n)=\{X_0>u_n\}$ and let $v_n$ be the corresponding scaling factor defined above.

\begin{definition}
We define the \emph{rare event point process} (REPP) by
counting the number of exceedances (or hits to $U(u_n)$) during the (re-scaled) time period $v_nJ\in\RR$, where $J\in\RR$. To be more precise, for every $J\in\RR$, set
\begin{equation*}
\label{eq:def-REPP} N_n(J):=\nn_{u_n}(v_nJ)=\sum_{j\in v_nJ\cap\N_0}\I_{X_j>u_n}.
\end{equation*}
\end{definition}

As will see below the REPP just defined converges in distribution to either to standard Poisson process or to a compound Poisson process $N$ with intensity $\theta$ and a geometric multiplicity d.f. For completeness, we define here what we mean by a compound Poisson process. (See \cite{K86} for more details).

\begin{definition}
\label{def:compound-poisson-process}
Let $T_1, T_2,\ldots$ be  an i.i.d. sequence of random variables with common exponential distribution of mean $1/\theta$. Let  $D_1, D_2, \ldots$ be another i.i.d. sequence of random variables, independent of the previous one, and with d.f. $\pi$. Given these sequences, for $J\in\RR$, set
$$
N(J)=\int \I_J\;d\left(\sum_{i=1}^\infty D_i \delta_{T_1+\ldots+T_i}\right),
$$ 
where $\delta_t$ denotes the Dirac measure at $t>0$.  Whenever we are in this setting, we say that $N$ is a compound Poisson process of intensity $\theta$ and multiplicity d.f. $\pi$.
\end{definition}
\begin{remark}
\label{rem:poisson-process}
In this paper, the multiplicity will always be integer valued which means that $\pi$ is completely defined by the values $\pi_k=\p(D_1=k)$, for every $k\in\N_0$. Note that, if $\theta=1$ and $\pi_1=1$, then $N$ is the homogenous  standard Poisson process and, for every $t>0$, the random variable $N([0,t))$ has a Poisson distribution of mean $t$. 
\end{remark}

\begin{remark}
\label{rem:compound-poisson}
At periodic points we will see that $\pi$ is actually a geometric distribution of parameter $\theta\in (0,1]$,  \ie $\pi_k=\theta(1-\theta)^{k-1}$, for every $k\in\N_0$. This means that, as in \cite{HV09}, here, the random variable $N([0,t))$ follows a P\'olya-Aeppli distribution, \emph{i.e.}:
$$
\p(N([0,t))=k)=\e^{-\theta t}\sum_{j=1}^k \theta^j(1-\theta)^{k-j}\frac{(\theta t)^j}{j!}\binom{k-1}{j-1},
$$
for all $k\in\N$ and $\p(N([0,t))=0)=\e^{-\theta t}$. 
\end{remark}

\begin{remark}
\label{rem:convergence-point-processes}
By \cite[Theorem~4.2]{K86}, the sequence of point processes $(N_n)_{n\in\N}$ converges in distribution to the point process $N$ iff the sequence of vector r.v.  $(N_n(I_1), \ldots, N_n(I_k))$ converges in distribution to $(N(I_1), \ldots, N(I_k))$, for every $k\in\N$ and all $I_1,\ldots, I_k\in\S$ such that $N(\partial I_i)=0$ a.s., for $i=1,\ldots,k$.
\end{remark}

A few examples of studies of REPPs in a dynamical context can be found in \cite{DenGorSha04, HirSauVai99, FreFreTod10, ChaCol13, FHN14}, and in a more probabilistic context in \cite{P91, AbaVer08, KifRap14}.

\subsection{Main results}

Let $X_0, X_1, \ldots$ be as in \eqref{eq:def-SP}, with $\varphi$ as specified above.

\begin{theorem}  Given $\z\in (0,1]$, consider the REPP $N_n$ defined above.  Then either
\begin{enumerate}
\item[(a)] $\z$ is not periodic and $N_n$  converges in distribution to a homogeneous Poisson process $N$ with intensity 1.
\item[(b)] $\z$ is periodic with period $p$ and $N_n$ converges in distribution to a compound Poisson process $N$ with intensity $\theta=1-\left| D(f^{-p})(\z)\right|$ and multiplicity distribution function $\pi$ given by
$
\pi_\kappa=\theta(1-\theta)^{\kappa-1},
$
for every $\kappa\in\N_0$.
\end{enumerate}
\label{thm:main}
\end{theorem}

\begin{theorem}
\label{thm:fixed-point}
For $\z=0$, consider the maximum function $M_n=M_n(x)$ defined above.
\begin{enumerate}
\item[(a)] Let $(u_n)_n=(u_n(\tau))_n$ be chosen as in \eqref{eq:un}, then $\p(M_n\le u_n)\to 1$ as $n\to \infty$ for any $\tau>0$.
\item[(b)] For $\alpha\in(0,\sqrt5-2)$ and $\tau\geq 0$, consider the sequence of thresholds $(u_n)_n=(u_n(\tau))_n$ satisfying:
\begin{equation}
\label{eq:new-un-1}
\lim_{n\to\infty}n\mu_\alpha\big(\left[x_n,\varphi^{-1}(u_n)\right)\big)=\tau,
\end{equation}
where $x_n$ is such that $x_n<\varphi^{-1}(u_n)$ and $f_\alpha(x_n)=\varphi^{-1}(u_n)$. Then we have 
$$\lim_{n\to\infty}\p(M_n\le u_n)\to e^{-\tau}.$$ 
\end{enumerate}
\label{thm:zero}
\end{theorem}

\begin{remark}
Observe that the normalising sequence $(u_n)_{n\in\N}$ given by \eqref{eq:new-un-1} does not satisfy condition \eqref{eq:un}. In fact, for such a sequence we have $\lim_{n\to\infty}n\mu_\alpha(X_0>u_n)=\infty$. This is coherent with \cite[Corollary~3.7.4]{LLR83}. In classical Extreme Value Theory (as well as in our Theorems~\ref{thm:main} and \ref{thm:induced poiss}), the normalising constants are usually taken to be linear families depending on a parameter $y$, \ie $u_n=a_n^{-1}y+b_n$, with $a_n>0$ for all $n\in\N$, so that the limiting distribution for such linear normalisation is an extremal type distribution with $\tau=\tau(y)$ assuming three different types, see \cite{F13} and Section~\ref{ssec:shapes} for more details. 
When the extremal index is equal to 0, the types of limiting law may be different from the expected ones if the stochastic process was independent, as noted in \cite[Page~69]{LLR83}. It is interesting to observe that the normalising sequences proposed here may lead to the same extremal limiting law as if the process was independent, or not (see Section~\ref{ssec:shapes}).
\label{rmk:un}
\end{remark}

\begin{remark}
We note that in the case $\z=0$, while it is possible to rescale the thresholds to recover an EVL as in Theorem~\ref{thm:zero} (2), the corresponding REPP remains degenerate.  This result will form part of a forthcoming work \cite{FFR15}.
\end{remark}

\begin{remark}
The choice $\alpha\in(0,\sqrt5-2)$ is for technical reasons. This upper bound on $\alpha$ depends on the rate of decay of correlations, on the corresponding parameter $\beta$ and also on $\epsilon$, determining the regularity of the density. We would expect the result to hold for at least $\alpha\in (0, 1/2)$ but the current knowledge about these maps does not allow it. The obstructions appear due to the restrictions on  the space of observable functions for which the statement  \eqref{eq:Holder-DC}  about decay of correlations holds. The underlying reasoning is a blocking argument which requires a mixing condition (see condition $\D_q(u_n)$ below) needed to give asymptotic independence between the blocks of random variables. This condition would follow immediately (for $\alpha<1/2$) if one could plug indicator functions into \eqref{eq:Holder-DC}. Since the existing information about decay of correlations requires the use of Banach spaces such as H\"older functions, one has to replace the indicator functions by suitable H\"older continuous approximations, which ultimately leads to further restrictions on $\alpha$. These are the same technical limitations experienced for example in \cite{HNT12} where $\alpha\in (0, \omega_0)$ for $\omega_0\approx 1/13$. 
\end{remark}

\subsection{Comments on history and strategy}
Before discussing our approach we introduce some notation.  For a dynamical system $f:\X\to \X$ and a subset $A\subset \X$, for $x\in \X$ define 
$$r_A(x):=\inf\{n\in \N:f^n(x)\in A\},$$ the \emph{first hitting time} to $A$.
Note that there is a connection with the behaviour of the variable $r_{U(u_n)}$ and our REPP since we can break that process down into a sequence of first hits to $U(u_n)$.  This gives a connection with our REPP and the asymptotics of $r_{U(u_n)}$, the \emph{Hitting Time Statistics} (HTS).  One basic difference is that here we are concerned with all hits to $U(u_n)$, not just the first.

Our main result for the case of $\p$-typical points and for periodic points in $(0,1)$ follows quickly from previous works, including works already mentioned above, and indeed in some of these papers mention MP  explicitly.  We also remark that some of the earliest works on HTS for dynamical systems considered the case of MP maps with $\alpha\ge 1$, see for example \cite{ColGal93, ColGalSch92, CamIso95}, with a focus on the behaviour at 0.  In these cases, the sets $A_n$ considered were formed from dynamically defined cylinder sets and the analysis was done at 1/2, the preimage of 0, so that finite measure sets could be used.  In this paper we consider  the case $\alpha\in (0, 1)$, so $f$ has an acip, and we also consider more general points and sets $A_n$.

We will first consider all points in $(0,1)$, using inducing methods.  This will require us to generalise the already very flexible result of \cite{HayWinZwe14} to point processes.  Finally we use the approach which goes back to Leadbetter \cite{L73} of proving some short range and long range recurrence conditions to prove that we have a degenerate law at 0 (the extremal index is 0).

\section{Induced point processes}

Here we aim to generalise \cite{HayWinZwe14} to point processes.  In that paper, they use \cite[Corollary 5]{Zwe07a} as one of their key tools.  In our, fairly analogous, setting we use \cite[Corollary 6]{Zwe07a} instead.  Note that previous results here include \cite[Theorem 2.1]{BruSauTroVai03}, where they proved that for balls around typical points, the HTS of first return maps are the same as that for the original map - they also remarked, without details, that this can be extended to successive return times.  Also in \cite{FreFreTod13}, we extended this idea to periodic points.  The strengths of the approach in  \cite{HayWinZwe14} to HTS are that it covers \emph{all} points, and that the proof is rather short.

We will give our result comparing the point process of the induced system to that coming from the original system in a general setting and then later apply this to our MP example.
In this section, we take a dynamical system $f:\X\to \X$ with an ergodic $f$-invariant probability measure $\mu$, choose a subset $Y\subset \X$, recalling that $r_Y:Y\to \N$ is the first return time to $Y$, consider $F=F_Y:Y\to Y$ to be the first return map $f^{r_Y}$ to $Y$ (note that $r_Y$ and thus $F$ may be undefined at a zero Lebesgue measure set of points which do not return to $Y$, but most of these points are not important, so we will abuse notation here).  Let $\mu_Y(\cdot)=\frac{\mu(\ \cdot \  \cap Y)}{\mu(Y)}$ be the conditional measure on $Y$.  By Kac's Theorem $\mu_Y$ is $F_Y$-invariant.

Setting $v_n^Y=1/\mu_Y(X_0>u_n)$, for the induced process $X_i^Y=\phi\circ F_Y^i$,
$$N_n^Y(J):=\nn_{u_n}^Y(v_n^YJ)=\sum_{j\in v_n^YJ\cap\N_0}\I_{X_j^Y>u_n}.$$

In keeping with \cite{HayWinZwe14}, we denote our inducing domain by $Y$.  Denote the speeded up return time by $r_{A, Y}$ (i.e., $r_{A, Y}(x)=\inf\{n\in \N:F_Y^n(x)\in A\}$)  and the induced measure on $Y$ by $\mu_Y$. For each $k\geq 2$ and $x\in\X$, we also define $r^k_{A}(x)=r_{A}\left(f^{r^1_{A}(x)+\ldots+r^{k-1}_{A}(x)}(x)\right)$ and, for $x\in Y$ and $A\subset Y$, the corresponding speeded up version $r^k_{A,Y}(x)=r_{A}\left(F_Y^{r^1_{A,Y}(x)+\ldots+r^{k-1}_{A,Y}(x)}(x)\right)$.
Moreover, for $\kappa>0$ and $I\in\mathcal S$, set $I^\eta:=\cup_{s\in I}B_{\eta}^+(s)$ where $B_\eta^+(s)=(s-\eta, s+\eta) \cap [0,\infty)$. For $J=\in\mathcal R$ such that $J=\cup_{j=1}^\ell I_j $, set $J^\eta=\cup_{j=1}^\ell I_j^\eta$.

\begin{theorem}
For every $J\in\mathcal R$, assume that $N(J^\eta)$ is continuous in $\eta$, for all small $\eta$. That is to say that if $J=\cup_{j=1}^\ell I_j $, then for every $k_1, \ldots, k_\ell$, the map $\eta\mapsto \mu_Y(N(I_1^\eta)\ge k_1, \ldots,N(I_\ell^\eta)\ge k_\ell)$ is continuous for $\eta\in [0, \eta_0)$. 
Then
$$N_n^Y\stackrel{\mu_Y}{\Longrightarrow} N \text{ as } n\to\infty \text{ implies } N_n \stackrel{\mu}{\Longrightarrow} N \text{ as } n\to\infty.$$
\label{thm:pp_ret_orig}
\end{theorem}

\begin{proof}
By \cite[Corollary 6]{Zwe07a}, for hitting times point processes such as $(N_n)_n$ and an ergodic reference measure $m$, if $P\ll m$ then $N_n\stackrel{P}{\Longrightarrow} N$ in $\M_p([0, \infty))$ implies $N_n\stackrel{Q}{\Longrightarrow} N$  in $\M_p([0, \infty))$ for any $Q\ll m$.  So replacing both $m$ and $Q$ with $\mu$ and replacing $P$ with $\mu_Y$ we see that 
for our sequence of processes, if $N_n\stackrel{\mu_Y}{\Longrightarrow} N$ in $\M_p([0, \infty))$, then $N_n\stackrel{\mu}{\Longrightarrow} N$  in $\M_p([0, \infty))$.  Thus, by Remark~\ref{rem:convergence-point-processes}, 
 it suffices to show that for every $J\in \RR$, such that $J=\cup_{j=1}^\ell I_j $,  and 
 all $k_1,\ldots,k_\ell\in \N$, we have
\begin{multline*}
\mu_Y(N_n^Y(I_1)\ge k_1,\ldots, N_n^Y(I_\ell)\ge k_\ell ) \stackrel{n\to\infty}{\longrightarrow} \mu_Y(N(I_1)\ge k_1, \ldots, N(I_\ell)\ge k_\ell )\\ \text{ implies } \mu(N_n(I_1)\ge k_1,\ldots,N_n(I_\ell)\ge k_\ell) \stackrel{n\to\infty}{\longrightarrow} \mu(N(I_1\ge k_1,\ldots, N(I_\ell\ge k_\ell).\end{multline*}
For $\delta>0$ and $M\in \N$, let 
$$E_M=E_M^\delta:=\left\{\left(\frac{1-\delta}{\mu(Y)}\right)j\le r_Y^j\le \left(\frac{1+\delta}{\mu(Y)}\right)j \text{ for all } j\ge M\right\} \text{ and } G_N:=\{r_{U(u_n), Y}\ge N\}.$$
As in \cite{HayWinZwe14}, $\mu_Y(G_N^c)\le N\mu_Y(U(u_n))\to 0$ as $n\to\infty$.  Also the ergodic theorem says that $\mu_Y((E_M^\delta)^c)\to 0$ as $M\to\infty$.  
Hence we may restrict our focus to $G_M\cap E_M^\delta$.

For $x\in E_M^\delta$, $${r_{U(u_n), Y}^k(x)}\left(\frac{1-\delta}{\mu(Y)}\right)\le r_{U(u_n)}^k(x)=\sum_{j=0}^{r_{U(u_n), Y}^k(x)-1}r_Y\circ F_Y^j(x) =r_Y^{r_{U(u_n), Y}^k(x)}(x)\le {r_{U(u_n), Y}^k(x)}\left(\frac{1+\delta}{\mu(Y)}\right) $$
We can deduce that for $x\in G_M\cap E_M^\delta$,
$$\mu(Y)r_{U(u_n)}^k(x) \in B_{\delta r_{U(u_n), Y}^k(x)}(r_{U(u_n), Y}^k(x)),$$
where we use the notation $B_\eps(y)=(y-\eps,y+\eps)$.

So if $r_{U(u_n), Y}^k(x)\in v_n^YJ$ then $\mu(Y)r_{U(u_n)}^k(x)\in v_n^YJ^\delta$ and so $r_{U(u_n)}^k(x)\in v_nJ^{\delta}$.  Therefore,\begin{multline*}
\mu_Y\left(\left\{N_n(I_1^{\delta})\ge k_1,\ldots,N_n(I_\ell^{\delta})\ge k_\ell\right\}\cap (E_M^\delta\cap G_M)\right) \ge\\ \mu_Y\left(\left\{N_n^Y(I_1)\ge k_1, \ldots,N_n^Y(I_\ell)\ge k_\ell\right\}\cap (E_M^\delta\cap G_M)\right).\end{multline*}
Setting $\delta':=\frac\delta{1+\delta}$, we also obtain that 
$$\frac1{\mu(Y)} r_{U(u_n), Y}^k(x) \in B_{\delta' r_{U(u_n)}^k(x)}(r_{U(u_n)}^k(x))$$
 for $x\in G_M\cap E_M^\delta$.  Analogously to above, this leads us to 
\begin{multline*}\mu_Y\left(\left\{N_n^Y(I_1^{\delta'})\ge k_1,\ldots,N_n^Y(I_\ell^{\delta'})\ge k_\ell\right\}\cap (E_M^\delta\cap G_M)\right) \ge \\ \mu_Y\left(\left\{N_n(I_1)\ge k_1,\ldots,N_n(I_\ell)\ge k_\ell\right\}\cap (E_M^\delta\cap G_M)\right).
\end{multline*}
So since $\eps, \delta>0$ were arbitrary, we are finished.
\end{proof}

\section{Application of inducing to Manneville-Pomeau}

In this section we prove our main theorem for all points $\z\in (0,1)$.

Let $\P$ be the \emph{renewal partition}, that is the partition defined inductively by $\cyl\in \P$ if $\cyl=[1/2, 1)$ or $f(\cyl)\in \P$.  Now let $Y\in \P$ and let $F_Y$ be the first return map to $Y$ and $\mu_Y$ be the conditional measure on $Y$.  It is well-known that $(Y, F_Y, \mu_Y)$ is a Rychlik system (see \cite{R83} or \cite[Section~3.2.1]{AFV15} for the essential information about such systems) and so the REPP is understood as in \cite[Corollary 3]{FreFreTod13}.  Hence by \cite{AFV15} we have the following theorem.

\begin{theorem}  Given $\z\in Y$, consider the REPP $N_n^Y$ defined above.  Then either
\begin{enumerate}
\item[(a)] $\z$ is not periodic and $N_n^Y\stackrel{\mu_Y}{\Longrightarrow}N$, where $N$ is a homogeneous Poisson process with intensity 1.
\item[(b)] $\z$ is periodic with period $p$ and $N_n^Y\stackrel{\mu_Y}{\Longrightarrow} N$, where $N$ is a compound Poisson process  with intensity $\theta=1-\left|D(F_Y^{-p})(\z)\right|$ and multiplicity d.f. $\pi$ given by \footnote{We note that there is an error in \cite[Theorem 1]{FreFreTod13}, propagated throughout the main results there: the $\kappa$ should be replaced by $\kappa-1$.}
$\pi_\kappa=\theta(1-\theta)^{\kappa-1}$,
for every $\kappa\in\N_0$.
\end{enumerate}
\label{thm:induced poiss}
\end{theorem}

For points in $Y\sm \cup_{n\ge 1} f^{-n}(0)$, this theorem is Proposition 3.2 of  \cite{AFV15}.  For the boundary points $\cup_{n\ge 1} f^{-n}(0)$, in the language of  \cite{AFV15}, any such point is called aperiodic non-simple.  Hence by Proposition 3.4(1) of that paper, we have a standard extremal index of 1 at all such points. Varying $Y$ means that we have considered all points in $(0,1)$.  So combining Theorems~\ref{thm:pp_ret_orig} and \ref{thm:induced poiss} completes the proof of Theorem~\ref{thm:main} for $\z\neq 0$.

\section{Analysis of the indifferent fixed point}

The tangency of the graph of the MP map with the identity map, creates an intensive clustering of exceedances of levels $(u_n)_{n\in \N}$, when they are chosen as in \eqref{eq:un}, that leads to the existence of an EI equal to 0, which leads to a degenerate limit distribution for $M_n$. However, if we choose the levels $(u_n)_{n\in \N}$ not in the classical way, but rather a sequence of lower thresholds, so that the exceedances that escape the clustering effect have more weight, then we can recover the existence of a non-degenerate distribution for the maxima.

The proof of an EI equal to 0 for the usual normalising sequences follows easily from the existing connections between Return Times Statistics (RTS), Hitting Times Statistics (HTS) and EVL, which we briefly recall in the next subsection. The proof of the existence of a non degenerate limit, under a different normalising sequence of thresholds, is more complicated and requires some new results from \cite{FFT15}, which we will recall below.

\subsection{The usual normalising sequences case}

For any $\zeta\in[0,1]$, let $B_\varepsilon(\zeta)=(\zeta-\varepsilon,\zeta+\varepsilon)\cap[0,1].$ 
Combining the main result from \cite{FreFreTod10} and  \cite{HLV07}, if there exists a non degenerate
d.f. $\tilde G$ such that  for all $t\ge 0$,
\begin{equation*}
\lim_{\varepsilon\to 0}\mu_{\alpha}\left(r_{B_\varepsilon(\zeta)}\leq \frac t{\mu_\alpha(B_\varepsilon(\zeta))} \ \middle\vert \ B_\varepsilon(\zeta)\right)=\tilde G(t),
\end{equation*} 
then for $G$ defined by
\begin{equation}
\label{eq:HTS-RTS}
G(t)=\int_0^t(1-\tilde G(s))\,ds,
\end{equation}
it can be shown that $H$ exists and equals $G$.

Let $U=[0,b)$ and $A=[a,b)$, where $a$ is such that $f(a)=b$, \ie $b=a+2^\alpha a^{1+\alpha}$. Using \eqref{eq:estimate-measure} we easily get
$\mu_\alpha(U)\sim_{c} a^{1-\alpha}+(1-\alpha)2^{\alpha}a+o(a)$ and $\mu_\alpha([0,a))\sim_{c} a^{1-\alpha}$.

Next we compute the RTS distribution, which we denote by $\tilde G(s)$. For $s\leq 0$, we easily have that $\tilde G(s)=0$, since $r_U\geq1$, by definition of hitting time. Let $s>0$ then
\begin{align*}
\tilde G(s)&=\lim_{b\to0} \mu_U\left(r_U \leq \frac s{\mu(U)}\right)=\lim_{b\to0} \frac1{\mu(U)}\mu\left(\left\{r_U \leq \frac s{\mu(U)}\right\}\cap U\right)\\
&\geq \lim_{b\to0} \frac {\mu(U\setminus A)}{\mu(U)}=\lim_{b\to0} \frac{\mu([0,a))}{\mu([0,b))}=1
\end{align*}
Then by \eqref{eq:HTS-RTS},
$G(t)=\int_0^t 1-\tilde G(s)ds =0$, which , by \cite{FreFreTod10}, corresponds to an EI equal to 0. Recall that $\bar H(\tau)=\e^{-\theta\tau}=1$, which means that, in this case, $H(\tau)=0$.

\subsection{Adjusted choice of thresholds}

In order to prove the existence of EVLs in a dynamical systems context, there are a couple of conditions on the dependence structure of the stochastic process that if verified allow us to obtain such distributional limits. These conditions are motivated by the conditions $D(u_n)$ and $D'(u_n)$ of Leadbetter but were adapted to the dynamical setting and further developed both in the absence of clustering,  such as in \cite{C01, FF08a, HNT12}, and in the presence of clustering in \cite{FreFreTod12}. Very recently, in \cite{FFT15}, the authors provided some more general conditions, called $\D(u_n)$ and $\D_q'(u_n)$, which subsumed  the previous ones and allowed them  to address both the presence ($q\geq1$) and the absence ($q=0$) of clustering. To distinguish these conditions the authors used a Cyrillic D to denote them. We recall these conditions here.

Given a sequence $(u_n)_{n \in \N}$ of real numbers satisfying \eqref{eq:un} and $q\in\N_0$, set
$$A_n^{(q)}:=\{X_0>u_n,X_1\leq u_n,\ldots, X_q\leq u_n\}.
$$

For $s,\ell\in\N$ and an event $B$, let
\begin{equation}
\label{eq:W-def}
\mathscr W_{s,\ell}(B)=\bigcap_{i=s}^{s+\ell-1} f^{-i}(B^c).
\end{equation}

\begin{condition}[$\D_q(u_n)$]\label{cond:D} We say that $\D(u_n)$ holds for the sequence $X_0,X_1,\ldots$ if, for every  $\ell,t,n\in\N$
\begin{equation}\label{eq:D1}
\left|\p\left(\A_n\cap
 \mathscr W_{t,\ell}\left(\A_n\right) \right)-\p\left(\A_n\right)
  \p\left(\mathscr W_{0,\ell}\left(\A_n\right)\right)\right|\leq \gamma(q,n,t),
\end{equation}
where $\gamma(q,n,t)$ is decreasing in $t$ and  there exists a sequence $(t_n)_{n\in\N}$ such that $t_n=o(n)$ and
$n\gamma(q,n,t_n)\to0$ when $n\rightarrow\infty$.
\end{condition}

For some fixed $q\in\N_0$, consider the sequence $(t_n)_{n\in\N}$ given by condition $\D_q(u_n)$ and let $(k_n)_{n\in\N}$ be another sequence of integers such that
\begin{equation}
\label{eq:kn-sequence}
k_n\to\infty\quad \mbox{and}\quad  k_n t_n = o(n).
\end{equation}

\begin{condition}[$\D'_q(u_n)$]\label{cond:D'q} We say that $\D'_q(u_n)$
holds for the sequence $X_0,X_1,\ldots$ if there exists a sequence $(k_n)_{n\in\N}$ satisfying \eqref{eq:kn-sequence} and such that
\begin{equation}
\label{eq:D'rho-un}
\lim_{n\rightarrow\infty}\,n\sum_{j=1}^{\lfloor n/k_n\rfloor}\p\left( \A_n\cap f^{-j}\left(\A_n\right)
\right)=0.
\end{equation}
\end{condition}

We note that, when $q=0$, condition $\D'_q(u_n)$ corresponds to condition $D'(u_n)$ from \cite{L73}.

Now let
\begin{equation}
\label{eq:OBrien-EI}
\vartheta=\lim_{n\to\infty}\vartheta_n=\lim_{n\to\infty}\frac{\p(\A_n)}{\p(U_n)}.
\end{equation}

From \cite[Corollary~2.4]{FFT15}, it follows that if the stochastic process $X_0, X_1,\ldots$ satisfies both conditions $\D_q(u_n)$ and $\D'_q(u_n)$ and the limit in \eqref{eq:OBrien-EI} exists then 
$$\lim_{n\to\infty}\p(M_n\leq u_n)= \e^{-\vartheta\tau}.$$

Now, we consider the fixed point $\z=0$. For every $n\in\N$, we require $\mu_\alpha(U_n)\sim \tau/n$.  Set $y_n$ to be such that $U_n=\{X_0>u_n\}=[0,y_n)$ and set $x_n\in U_n$ so that $f_\alpha(x_n)=y_n$, \ie $y_n=x_n+2^\alpha x_n^{1+\alpha}$. Using \eqref{eq:estimate-measure} we easily get
\begin{align}
\mu_\alpha(U_n)&\sim_{c} x_n^{1-\alpha}+(1-\alpha)2^{\alpha}x_n+o(x_n)\label{eq:Un-estimate}\\
\mu_\alpha([0,x_n))&\sim_{c} x_n^{1-\alpha}\label{eq:Un-1-estimate}\\
\mu_\alpha([x_n, y_n))&\sim_{c} (1-\alpha)2^{\alpha}x_n+o(x_n)\label{eq:Qn-estimate}
\end{align}

Now, since we are assuming that $\mu_\alpha(U_n)\sim \tau/n$, then $x_n\sim_{c} 1/n^{1/(1-\alpha)}$. Observe that $\mu_\alpha(U_n\cap f^{-1}_\alpha(U_n))=\mu_\alpha([0,x_n))\sim_{c} x_n^{1-\alpha}\sim_{c} 1/n$. Hence, if we consider $q=0$, the periodicity of $\zeta$ implies that $\D'_q(u_n)$ does not hold since 
$$
n\sum_{j=1}^{\lfloor n/k_n\rfloor}\p\left( U_n\cap f^{-j}\left(U_n\right)\right)\geq n\mu_\alpha(U_n\cap f^{-1}_\alpha(U_n))>0,
$$
for all $n\in\N$. Hence, here, given that $\zeta$ is a periodic point of period $1$ the natural candidate for $q$ is $q=1$. From here on we always assume that $q=1$. 

In this case, $\A_n=[x_n, y_n)=:Q_n$. However, if we plug \eqref{eq:Qn-estimate} and \eqref{eq:Un-estimate} into \eqref{eq:OBrien-EI}, we obtain that $\vartheta=0$, which means that the natural candidate for a limit distribution for $\mu_\alpha(M_n\leq u_n)$ is degenerate.

The problem is that the indifferent fixed point creates too much dependence. In \cite{FreFreTod12}, under a condition called $S\!P$, we have seen that when $\zeta$ is periodic, the probability of having  no entrances in $U_n$, among the first $n$ observations, is asymptotically equal to the probability of having no entrances in $Q_n$, among the first $n$ observations, \ie 
$$
\lim_{n\to\infty}\p(M_n\leq u_n)=\lim_{n\to\infty}\p(\mathscr W_{0,n}(U_n))=\lim_{n\to\infty}\p(\mathscr W_{0,n}(Q_n)).
$$

In \cite{FFT15}, it was shown that it is possible to replace $U_n$ by $Q_n$ even without the $S\!P$ condition (see \cite[Proposition~2.7]{FFT15}). Making use of this upgraded result, we can now change the normalising sequence of levels $(u_n)_{n\in\N}$ so that we can still obtain a non-degenerate limit for $\p(M_n\leq u_n)$. To understand the need to change the normalising sequence in order to obtain a non-degenerate limit, recall that condition \eqref{eq:un} guaranteed that $M_n$ was normalised by a sequence of levels that kept the average of exceedances among the first $n$ observations  at an (almost) constant value $\tau>0$. When $\vartheta>0$, condition \eqref{eq:un} also guarantees that the average number of entrances in $Q_n$ among the first $n$ observations is kept at an (almost) constant value $\theta\tau>0$. Here, since $\vartheta=0$, we need to change $u_n$ so that the average number of entrances in $Q_n$ is controlled, \ie
\begin{equation}
\label{eq:new-un}
\lim_{n\to\infty} n\p(\A_n)=\tau>0.
\end{equation}

From equations (2.15) and (2.16) from \cite{FFT15} one gets:
\begin{align}
\label{eq:error1}
\left|\p(\mathscr W_{0,n}(\A_n))-\left(1-\left\lfloor \frac n{k_n}\right\rfloor\p(\A_n)\right)^{k_n}\right|\leq &  2k_nt_n\p(U_n)+ 2n\sum_{j=1}^{\lfloor n/k_n\rfloor-1}\p\left(\A_n \cap f^{-j}\A_n\right)\nonumber\\&+\gamma(q,n,t_n)
\end{align}
Note that since by \eqref{eq:new-un} we have $\lim_{n\to\infty}\left(1-\left\lfloor \frac n{k_n}\right\rfloor\p(\A_n)\right)^{k_n}=\e^{-\tau}$, then if both conditions $\D_q(u_n)$ and $\D'_q(u_n)$ hold, then all the terms on the left of \eqref{eq:error1} vanish, as $n\to\infty$, and consequently:
\begin{equation}
\label{eq:conclusion}
\lim_{n\to\infty}\p(M_n\leq u_n)=\lim_{n\to\infty}\p(\mathscr W_{0,n}(\A_n))=\e^{-\tau}.
\end{equation}
 
Hence, in order to show that we can still obtain a non-degenerate limiting law for the distribution of $M_n$ when $\zeta=0$, we start by taking a sequence $(u_n)_{n\in\N}$ so that \eqref{eq:new-un} holds. Note that this implies that by \eqref{eq:Qn-estimate} and \eqref{eq:Un-estimate} we have that $x_n\sim_{c} 1/n$ and $\mu_\alpha(U_n)\sim_{c} 1/n^{1-\alpha}$. In particular, this means that $\lim_{n\to\infty}n \mu_\alpha(U_n)=\infty$, which contrasts with the usual case where condition \eqref{eq:un} holds. 

To prove the existence of the limit in \eqref{eq:conclusion} we need to verify conditions $\D_q(u_n)$ and $\D'_q(u_n)$, where $q=1$. We start by the latter, which is more complicated. 


\subsubsection{Proof of $\D'_q(u_n)$}

We will next focus on the proof of $\D'_q(u_n)$ in the case of part (2) of Theorem~\ref{thm:zero}.  That is, $(x_n)_n$ will be chosen so that $x_n\sim_{c} 1/n$, as described above.  Later we will note that we can change $(x_n)_n$ to recover a degenerate law as in part (1) of that theorem.

We have to estimate the quantity
$$
\Delta'_n:=n\sum_{j=1}^{[n/k_n]}\mu_{\a}(Q_n\cap f^{-j}Q_n)
$$
where $Q_n=[x_n, y_n),$ for $x_n\sim_{c} \frac1n$ and $y_n=f(x_n).$ We follow the proof of \cite[Lemma 3.5]{HirSauVai99}. By denoting by $P$ the transfer operator and by $\tau_n\in \N$ the {\em first return time} of the set $Q_n$ into itself, we have:
$$
\Delta'_n\le n\  [n/k_n]\mu_{\alpha}(Q_n) \sup_{j=\tau_n, \dots, [n/k_n]}\sup_{Q_n}\frac{P^j({\bf 1}_{Q_n}h)}{h}
$$
where $h$ is the density of $\mu_{\a}.$ In order to compute $P^{\tau_n}({\bf 1}_{Q_n}h)$ we need to know how many branches of $f^{\tau_n}$ have their domain intersecting $Q_n.$ If $\xi_0$ is the original partition into the sets $[0,1/2), [1/2, 1]$, we denote with $\xi_k$ the join $\xi_k:=\xi_0\vee f^{-1}\xi_0\vee \cdots \vee f^{k-1}\xi_0.$

We begin by observing that $Q_n$ contains at most one boundary point of the partition $\xi_{\tau_n-1},$ otherwise one point of $Q_n$ should be sent into the same set, being $f^{\tau_n-1}$ onto on each domain of injectivity. Then when we move to $\xi_{\tau_n},$ the interval $Q_n$ will be covered by at most $4$ cylinders of monotonicity of the partition $\xi_{\tau_n}.$ By denoting them from left to right by $C_{\tau_n, 1},\dots, C_{\tau_n, 4}$ we have
$$
P^{\tau_n}({\bf 1}_{Q_n}h)=\sum_{i=1}^4\frac{h\circ f^{-\tau_n}_i {\bf 1}_{f^{\tau_n}_iQ_n}}{Df^{\tau_n}\circ f^{-\tau_n}_i}
$$
where $f^{\tau_n}_i$ denotes the  branch of $f^{\tau_n}$ restricted to $C_{\tau_n, i}.$
Notice that the density is computed in $Q_n$ whose left boundary point is $1/n$, so $h$ is bounded from above by a constant times $n^{\a}.$ We have now to estimate  the derivative $Df^{\tau_n}\circ f^{-\tau_n}_i$ on the sets $Q_n\cap C_{\tau_n, i}$.
Let us define $r_m$ as the $m$-left preimage of $1$, $r_m:=f^{-m}_1(1)$ and define $m(n)$ as $r_{m(n)}\le x_n \le r_{m(n)-1}.$ Then the interval $[x_n, y_n)$ will intersect the two cylinders $(r_{m(n)}, r_{m(n)-1})$ and $(r_{m(n)-1}, r_{m(n)-2})$ and the first return of $Q_n$ will be larger than the first returns of those two cylinders; on the other hand the first return of $(r_{m(n)}, r_{m(n)-1})$ is $m(n)$. The derivative $Df^{\tau_n}$ will be computed at a point $\iota_n$ which will be in one of those two cylinders; suppose without any restriction that $\iota_n\in (r_{m(n)}, r_{m(n)-1})$.  Since we need to bound from below the derivatives, we begin to replace $Df^{\tau_n}(\iota_n)$ with $Df^{m(n)}(\iota_n);$  then we observe that the map $f^{m(n)}: [r_{m(n)}, r_{m(n)-1}]\rightarrow [0,1]$ is onto and we use the distortion bound given, for instance, in [LSY, Lemma 5] which states that there exists a constant $C$ such that for any $m\ge 1$ and any $x,y\in [r_{m}, r_{m-1}]$ we have $\left |\frac{Df^m(x)}{Df^m(y)}\right|\le C$. We finally note that $m(n)\sim_{c} n^{\a}.$ This implies immediately  that $$\frac{1}{Df^{m(n)}(\iota_n)}\le C|r_{m(n)-1}-r_{m(n)}|\sim_{c} C\frac{1}{m(n)^{\frac{1}{\a}+1}}\sim_{c} C\frac{1}{n^{1+\a}}.$$

 Consequently ($C$ will continue to denote a constant which could vary from one bound to another)
$$
P^{\tau_n}({\bf 1}_{Q_n}h)\sim_{c} \frac{1}{n}
$$
We now continue as in [HSV] by getting for the other powers of the transfer operator:
$$
\frac{P^j({\bf 1}_{Q_n}h)}{h}\le \frac{P^{j-\tau_n}{\bf 1}}{h}\sup P^{\tau_n}({\bf 1}_{Q_n}h)\le \frac{P^{j-\tau_n}\frac{h}{\inf \ h}}{h}\sup P^{\tau_n}({\bf 1}_{Q_n}h)\le \frac{C}{\inf \ h}\frac{1}{n}
$$
 and finally
$$
\Delta'_n\le n \ [n/k_n]\mu_{\alpha}(Q_n) \frac{C}{\inf \ h}\frac{1}{n}
$$

We now know that $\mu_{\alpha}(Q_n)\sim_{c} \frac{1}{n};$ hence
$$
\Delta'_n\le C n \ [n/k_n] \mu_{\alpha}(Q_n)^2 \frac{ 1}{\mu_{\alpha}(Q_n)}\frac{1}{n}\sim_{c} [n^2 \ \mu_{\alpha}(Q_n)^2] \ \frac{1}{k_n}.
$$

So letting $n\to \infty$, we see that $\D_q'(u_n)$ holds.


\subsubsection{Proof of $\D_q(u_n)$}

This follows since, as in \eqref{eq:Holder-DC}, we have decay of correlations of H\"older functions against bounded measurable functions and condition $\D_q(u_n)$ was designed to follow from sufficiently fast decay of correlations, as shown in \cite[Proposition 5.2]{Fre13}. In order to compute the required rate of decay of correlations, which will impose a restriction on the domain of the parameter $\alpha$, we recall here the above-mentioned result so that we can follow the computations closely.
\begin{proposition}[{\cite[Proposition 5.2]{Fre13}}]
\label{prop:Holder-Dp}
Assume that $\X$ is a compact subset of $\R^d$ and $f:\X\to\X$ is a system with an acip $\p$, such that $\frac{d\p}{\text{dLeb}}\in L^{1+\epsilon}$.  Assume, moreover, that the system has decay of correlations for all $\phi\in\mathcal H_\beta$ against any $\psi\in L^{\infty}$ so that there exists some $C>0$ independent of $\phi,\psi$ and $t$, and a rate function 
$\varrho:\N\to\R$ such that 
\begin{equation}
\label{eq:Holder-DC-1}
\left| \int\phi\cdot(\psi\circ f^t)d\p-\int\phi d\p\int\psi
d\p\right|\leq C\|\phi\|_{\mathcal H_\beta}\|\psi\|_\infty \varrho(t),
\end{equation}
 and $n^{1+\beta(1+\max\{0,(\epsilon+1)/\epsilon-d\}+\delta)}\varrho(t_n)\to0$, as $n\to\infty$ for some $\delta>0$ and $t_n=o(n)$.
Then condition $\D_q(u_n)$ holds.
\end{proposition}  

\begin{remark}
We note that during the proof, in order to obtain the condition on the rate of decay of correlations, it is assumed that $\p(\A_n)\sim_{c} 1/n$.  
\end{remark}

Observe that since we are working in dimension $1$, which means $d=1$, then $\max\{0,(\epsilon+1)/\epsilon-d\}=1/\epsilon$. Also, from \cite{H04}, we may assume that the decay of correlations is written for Lipschitz functions, which allows us to take $\beta=1$. Hence, for condition $\D_q(u_n)$ hold, we need that the rate of decay of correlations $\varrho$ is sufficiently fast so that there exists some $\delta>0$ such that
\begin{equation}
\label{eq:D-estimate}
\lim_{n\to\infty}n^{2+1/\epsilon+\delta}\varrho(t_n)=0,
\end{equation}
where $t_n=o(n)$. From \eqref{eq:estimate-measure}, in order that the density $h_\alpha\in L^{1+\epsilon}$, we need that $\epsilon<1/\alpha-1$. Since by \eqref{eq:Holder-DC}, we have that $\varrho(t)=t^{-(1/\alpha-1)}$, then by \eqref{eq:D-estimate} it is obvious that we must have $\alpha<1/2$, which implies that $1/\epsilon<\alpha+2\alpha^2$. Taking $t_n=n^{1-\alpha}$, we obtain:
$$
n^{2+1/\epsilon+\delta}\left(n^{1-\alpha}\right)^{-(1/\alpha-1)}=n^{2+1/\epsilon+\delta}n^{-1/\alpha+2-\alpha}<n^{4+2\alpha^2+\delta-1/\alpha}.
$$
Hence, if $\alpha<\sqrt5-2$ we can always find $\delta>0$ so that \eqref{eq:D-estimate} holds and consequently condition $\D_q(u_n)$ is verified.

\subsubsection{Linear normalisation and extremal types limiting laws}\label{ssec:shapes}
As discussed in Remark~\ref{rmk:un}, in classical extreme value theory (independent case), typically, the limiting laws are obtained for linear normalising sequences $(a_n)_{n\in\N}, (b_n)_{n\in\N}$, with $a_n>0$ so that
$$
\lim_{n\to\infty}\p(a_n(M_n-b_n)\leq y)=\e^{-\tau(y)},
$$
where $\tau(y)$ is of one of three types specified in \cite[Equation (2.8)]{F13}, where the tail of the distribution function of $X_0$ determines the type of limit. Note that $u_n=y/a_n+b_n$ and equation \eqref{eq:un} imply that in the usual applications to dynamical systems, when no clustering occurs, the type of limit law is determined by the shape of $\varphi$ near $\zeta$ as well as the shape of the invariant density at that point. 

We note that the stochastic process defined as in \eqref{eq:def-SP}, with $\zeta=0$, may or may not have the same limiting extremal distribution  (under linear normalisation) as the corresponding independent sequence, i.e., consider $Z_0, Z_1,\ldots$ independent and identically distributed (i.i.d.) and such that $Z_0$ as the same distribution function as $X_0$.

To illustrate the possibility of having different extremal limiting distributions, we give an example of $\phi$ for which $\p(a_n(M_n-b_n)\leq y)\to H(y)$ and $\p(a^*_n(M^*_n-b^*_n)\leq y)\to H^*(y)$ but $H(y)\neq H^*(y)$, where $M_n^*=\max\{Z_0,\ldots, Z_{n-1}\}$. 

Let $\varphi=1-x$. Then $\varphi^{-1}(x)=1-x$. Using \eqref{eq:new-un} to define $u_n$, having in mind \eqref{eq:Qn-estimate} and recalling that $Q_n=[x_n,y_n)$, where $y_n=\varphi^{-1}(u_n)$ and $x_n<y_n$ is such that $f_\alpha(x_n)=y_n$, we see that for a well chosen constant $c>0$, we can take $u_n=1-c\tau/n$. Using Theorem~\ref{thm:zero} and writing $y=-\tau$, we get 
$$
\lim_{n\to\infty}\mu_{\alpha}(M_n\leq 1-c\tau/n)=\e^{-\tau}\Longleftrightarrow \lim_{n\to\infty}\mu_{\alpha}\left(\frac{n}{c}(M_n-1)\leq y\right)=\e^{-(-y)},
$$
which means we have Weibull limiting distribution with exponent equal to $1$ and normalising sequences given by $a_n=\frac{n}{c}$ and $b_n=1$.

If we are now to determine the limiting law for the sequence the i.i.d. sequence $Z_0, Z_1,\ldots$, we have to use equation \eqref{eq:un} to define $u_n$. Now since $U_n=\{X_0>u_n\}=[0,y_n)$, where $y_n=\varphi^{-1}(u_n)$, then by \eqref{eq:Un-1-estimate}, for a well chosen $c^*>0$, we can take $u_n=1-\left(\frac{c^*\tau}{n}\right)^{1/(1-\alpha)}$. Then using \cite[Theorem~1.5.1]{LLR83} and letting $y=-\tau^{1/(1-\alpha)}$, we would get
$$
\lim_{n\to\infty}\P\left(M^*_n\leq 1-\left(\frac{c^*\tau}{n}\right)^{1/(1-\alpha)}\right)=\e^{-\tau}\Leftrightarrow \lim_{n\to\infty}\p\left(\left(\frac{n}{c^*}\right)^{1/(1-\alpha)}(M^*_n-1)\leq y\right)=\e^{-(-y)^{1-\alpha}},
$$
which means we have Weibull limiting distribution with exponent equal to $1-\alpha$ and normalising sequences given by $a_n=\left(\frac{n}{c^*}\right)^{1/(1-\alpha)}$ and $b_n=1$.

The same computations would allow us to verify that if $\varphi(x)=-\log(x)$ then both the extremal limiting distributions under linear normalisation for both $M_n$ and $M_n^*$ would be the Gumbel distribution, \ie $H(y)=H^*(y)=\e^{-\e^{-y}}$.

\bibliographystyle{amsalpha}
\bibliography{jorge_dich2}

\end{document}